\documentclass{scrartcl}

\usepackage[latin1]{inputenc} 
\usepackage[ngerman, english]{babel}
\usepackage{amsmath,amsfonts,amsthm}
\usepackage{cite}
\usepackage{palatino}

\newcommand{\be}{\begin{equation}}
\newcommand{\ee}{\end{equation}}
\newcommand{\R}{\mathbb{R}}

\newcommand{\tp}{\text{T}}

\newtheorem{theorem}{Theorem}
\newtheorem{corollary}[theorem]{Corollary}
\theoremstyle{definition}
\newtheorem{remark}[theorem]{Remark}

\deffootnote[1em]{1em}{1em}{\textsuperscript{\thefootnotemark}}

\begin{document}
\title{Taylor Expansion of \\homogeneous functions}
\author{Joachim Paulusch and Sebastian Schlütter}
\date{July 7, 2021, this version: \today}
\maketitle
\begin{abstract}
We derive the Taylor polynomial of a function which is $m$-times continuously differentiable and positive homogeneous of order $m$. The Taylor polynomial in $a$ for $f(b)$ of order $m$ in general is a polynomial of order $m$ in $b-a$. If the given function is positive homogeneous of order $m$, the Taylor polynomial is a polynomial in $b$ rather than $b-a$, and the order of all terms is $m$. The result can be applied to powers of homogeneous functions of order $1$ as well.

The result was stated without proof by the International Actuarial Association in 2004.
\end{abstract}
We consider Taylor polynomials of homogeneous functions. Derivatives of higher order are denoted by
\be 
d^m f(x;t) = \sum_{k \in \{1,\ldots, n\}^m} D_k f(x) t_{k_1} \cdots t_{k_m}.
\ee
Let $U \subseteq \R^n$ be open and $f:U \to \R$ have continuous partial derivatives of order $m$ in $U$. Let $a, b \in U$ and the line segment connecting $a$ and $b$ in $U$. The Taylor polynomial for $f(b)$ of order $m$ at $a$ is given by 
\be 
T^{(m)}f(a; b-a) = f(a) + \sum_{k=1}^m \frac{1}{k!} d^k f(a; b-a) .
\ee
\begin{theorem} \label{TaylorThm}
Let $U \subseteq \R^n$ be open and $f:U \to \R$ be positive homogeneous\footnote{I.e. $f(\lambda x) = \lambda^m f(x)$ for all $x\in U$, $\lambda > 0$ with $\lambda x \in U$. This gives rise to a unique extension of $f$ to the set $\{ \lambda x\,|\, x \in U, \lambda >0\}$. In case $0 \in U$, this extends $f$ to $\R^n$.} of degree $m$ and have continuous partial derivatives of order $m$ in $U$. Let $a, b \in U$ and the line segment connecting $a$ and $b$ in $U$. The Taylor polynomial for $f(b)$ of order $m$ at $a$ is given by 
\be 
T^{(m)}f(a; b-a) = \frac{1}{m!} \, d^m f(a;b) .
\ee
\end{theorem}
This means that the Taylor polynomial is a polynomial in $b$ rather than $b-a$ and the order of all terms is $m$, i.e. a homogeneous polynomial of order $m$.  

Theorem \ref{TaylorThm} is stated without proof in \cite[p. 170]{IAA2004}. A proof for $m=2$ is given in \cite[p. 564]{Sandstroem2016}.

Theorem \ref{TaylorThm} immediately implies
\begin{corollary} \label{TaylorCor}
Let $U \subseteq \R^n$ be open and $f:U \to \R$ be positive homogeneous of degree $1$ and have continuous partial derivatives of order $\ell$ in $U$. Let $a, b \in U$ and the line segment connecting $a$ and $b$ in $U$. Then, for all $1 \leq m \leq \ell$, we have
\be 
T^{(m)}f^m (a; b-a) = \frac{1}{m!} \, d^m f^m (a;b) .
\ee
\end{corollary}
\begin{remark}
Note that in case $0 \in U$ it can be shown that a homogeneous function $f$ of order $m$ actually {\em is} a homogeneous polynomial of order $m$, cf. \cite[p. 175]{Lang1987}. A typical example not covered by this but Theorem \ref{TaylorThm}, or Corollary \ref{TaylorCor} respectively, is
\be f(x) = \sqrt{x^\tp R \,x},\ee
wherein $R$ is a matrix and $x^\tp R \,x > 0$. This can be applied to models for the aggregation of risk, c.f. \cite{Paulusch2021}.
\end{remark}
\begin{proof}[Proof of Theorem \ref{TaylorThm}]
To facilitate the handling of many dimensions we will use Einstein's summation convention, i.e. we write 
\be 
\sum_{k \in \{1,\ldots, n\}^m} D_k f(x) t_{k_1} \dots t_{k_m} = D_{\mu_1 \dots \mu_m} f(x) t^{\mu_1} \dots t^{\mu_m}. 
\ee
Note that Euler's Theorem on homogeneous functions \cite[p. 6]{Tasche1999} implies
\be \label{EulerEinsteinEqn}
D_{\mu_1\dots \mu_k} f(a) a^{\mu_k} = \left( m - (k-1) \right) D_{\mu_1\dots \mu_{k-1}} f(a), 
\ee
as $D_{\mu_1\dots \mu_{k-1}} f$ is homogeneous of order $\left( m - (k-1) \right)$. Moreover, by continuously differentiability, we have the symmetry
\be 
D_{\mu_1\dots \mu_k} f(a) = D_{\mu_{\sigma(1)}\dots \mu_{\sigma(k)}} f(a), 
\ee
for all $1 \leq k \leq m$ and permutation $\sigma$. Taylor's formula reads in Einstein notation
\be 
T^{(m)}f(a; b-a) = f(a) +  \sum_{k=1}^m \frac{1}{k!} D_{\mu_1\dots \mu_k} f(a) (b-a)^{\mu_1} \cdots (b-a)^{\mu_k}.
\ee
We use \eqref{EulerEinsteinEqn} to write 
\be 
D_{\mu_1\dots \mu_k} f(a)  = \frac{1}{(m-k)!} D_{\mu_1\dots \mu_m} f(a)  a^{\mu_{k+1}} \cdots a^{\mu_m}, 
\ee
that is (using the symmetry of the derivatives)
\be 
\begin{split}
& T^{(m)}f(a; b-a) = \frac{1}{m!} D_{\mu_1\dots \mu_m} f(a) \sum_{k=0}^m {m \choose k} (b-a)^{\mu_1} \cdots (b-a)^{\mu_k} a^{\mu_{k+1}} \cdots a^{\mu_m} \\
& \quad = \frac{1}{m!} D_{\mu_1\dots \mu_m} f(a) \sum_{k=0}^m {m \choose k} a^{\mu_{k+1}} \cdots a^{\mu_m}
\sum_{p=0}^k (-1)^p {k \choose p} b^{\mu_1} \cdots b^{\mu_{k-p}} a^{\mu_{k-p+1}} \cdots a^{\mu_k} \\
& \quad = \frac{1}{m!} D_{\mu_1\dots \mu_m} f(a) \sum_{k=0}^m \sum_{p=0}^k (-1)^p {m \choose k} {k \choose p} b^{\mu_1} \cdots b^{\mu_{k-p}}a^{\mu_{k-p+1}} \cdots  a^{\mu_m}.
\end{split}
\ee
To collect powers of $b$, we write $q = k-p$ and obtain
\be \label{PowerEqn}
\begin{split}
& T^{(m)}f(a; b-a) = \frac{1}{m!} D_{\mu_1\dots \mu_m} f(a) \sum_{q=0}^m \sum_{k=q}^m (-1)^{k-q} {m \choose k} {k \choose k-q} b^{\mu_1} \cdots b^{\mu_q} a^{\mu_{q+1}} \cdots  a^{\mu_m} \\
& \quad  = \frac{1}{m!} D_{\mu_1\dots \mu_m} f(a) \sum_{q=0}^m  b^{\mu_1} \cdots b^{\mu_q} a^{\mu_{q+1}} \cdots  a^{\mu_m} \sum_{k=q}^m (-1)^{k-q} {m \choose k} {k \choose k-q}.
\end{split}
\ee
For $q = m$, the inner sum over $k$ equals 1. So let $0 \leq q \leq m-1$. We write $p = k-q$, $r = m-q$ and derive for $1 \leq r \leq m$ 
\be
\begin{split}
&\sum_{k=q}^m (-1)^{k-q} {m \choose k} {k \choose k-q} = \sum_{p=0}^{m-q} (-1)^p {m \choose p+q} {p+q \choose p} \\
& \quad = \sum_{p=0}^{r}(-1)^p{m \choose p+m-r} {p+m-r \choose p} = \sum_{p=0}^{r}(-1)^p \frac{m!}{p! \,(m-r)! \,(r-p)!} \\
& \quad = {m \choose r} \sum_{p=0}^{r}(-1)^p {r \choose p} = {m \choose r} (1-1)^r = 0.
\end{split}
\ee
Going back to equation \eqref{PowerEqn}, this implies
\be
T^{(m)}f(a; b-a) = \frac{1}{m!} D_{\mu_1\dots \mu_m} f(a) b^{\mu_1} \cdots b^{\mu_m}, 
\ee
as was to be shown.
\end{proof}


\end{document}